\documentclass[11pt]{amsart}
\usepackage[english]{babel}
\usepackage{verbatim}
\usepackage{amsmath,latexsym,amssymb,verbatim}
\usepackage{graphicx}
\usepackage{color}
\usepackage{ulem}	
\usepackage{caption}
\usepackage{subcaption}
\usepackage{hyperref}

\newtheorem{lema}{Lemma}[section]
\newtheorem{theo}[lema]{Theorem}
\newcounter{teoremaganso}

\newtheorem{prop}[lema]{Proposition}

\newtheorem{coro}[lema]{Corollary}
\theoremstyle{definition}

\theoremstyle{remark}
\newtheorem{rema}[lema]{Remark}

\newtheorem{exam}[lema]{Example}

\def\sideremark#1{\ifvmode\leavevmode\fi\vadjust{\vbox to0pt{\vss 
			\hbox to 0pt{\hskip\hsize\hskip1em           
				\vbox{\hsize2.5cm\tiny\raggedright\pretolerance10000
					\noindent #1\hfill}\hss}\vbox to8pt{\vfil}\vss}}}%

\title[Rational solutions of Abel Differential Equations]{Rational solutions of Abel Differential Equations}

\author{J.L. Bravo}
\address{Departamento de Matematicas, Universidad de Extremadura, 06071 Badajoz, Spain}
\email{trinidad@unex.es}
\author{L.A. Calder\'on*}
\address{Departamento de Matematicas, Universidad de Extremadura, 06071 Badajoz, Spain}
\email{lucalderonp@unex.es}
\author{M. Fern\'andez}
\address{Departamento de Matematicas, Universidad de Extremadura, 06071 Badajoz, Spain}
\email{ghierro@unex.es}
\author{I. Ojeda}
\address{Departamento de Matematicas, Universidad de Extremadura, 06071 Badajoz, Spain}
\email{ojedamc@unex.es}
\thanks{* Corresponding author}
\thanks{
The authors are partially supported by Junta de Extremadura/FEDER grant number IB18023. JLB, LAC and MF are also partially supported by Junta de Extremadura/FEDER grant number GR21056. JLB and MF are also partially supported by MINECO/FEDER grant number PID2020-118726GB-I00. IO is also partially supported by Grant PGC2018-096446-B-C21 funded by MCIN/AEI/10.13039/501100011033, by ``ERDF A way of making Europe", and by Grant GR21055 funded by Junta de Extremadura (Spain)/FEDER funds.}

\begin{document}

\begin{abstract}
 We study the rational solutions of the Abel equation $x'=A(t)x^3+B(t)x^2$ where $A$ and $B \in \mathbb{C}[t]$. We prove that if $\deg(A)$ is even or $\deg(B)>(\deg(A)-1)/2$ then the equation has at most two rational solutions. For any other case, an upper bound on the number of rational solutions is obtained. Moreover, we prove that if there are more than $(\deg(A)+1)/2$ rational solutions then the equation admits a Darboux first integral.
\end{abstract}

\maketitle

\medskip
\noindent
\textcopyright{} 2022 Elsevier Inc.
This accepted manuscript is made available under the
\href{https://creativecommons.org/licenses/by-nc-nd/4.0/}
{CC BY-NC-ND 4.0 license}.

\noindent
Published version:
\href{https://doi.org/10.1016/j.jmaa.2022.126368}
{doi:10.1016/j.jmaa.2022.126368}.

\medskip

\section{Introduction}

\noindent
Abel differential equations of the form
\begin{equation}
	x'=A(t)x^3+B(t)x^2,\label{eq:Abel}
\end{equation}
where $A$ and $B$ are polynomials or trigonometric polynomials, have been extensively studied for its own intrinsic interest~\cite{G} and 
for their relationship with some real-world phenomena (see e.g.~\cite{BNP} and references therein), but also as toy models 
for studying important problems in the qualitative theory of differential equations.

One of these problems is the Smale--Pugh problem~\cite{S} of bounding the number of limit cycles (closed solutions isolated in the set of closed solutions),
which can be seen as a special case of Hilbert's 16th problem. Lins Neto~\cite{LN} 
proved that there is no upper bound on the number of limit cycles of~\eqref{eq:Abel}
for $A$ and $B$ polynomials or trigonometric polynomials.

Another important problem is the Poincar\'e Centre-Focus Problem which, in this setting, 
asks when the solutions of~\eqref{eq:Abel} in a neighbourhood of the solution $x(t)\equiv 0$
are all closed. This problem for~\eqref{eq:Abel} was proposed by Briskin, Fran\c{c}oise, and 
Yondim~\cite{BFY1,BFY2}. A special case when this holds 
is when the composition condition holds, as will be detailed later. 
This condition seems to play a key role in the problem, as shown in \cite{BRY,P}.

A natural problem in this context is to study polynomial or trigonometric polynomial
solutions (depending on whether $A$ and $B$ are polynomials or trigonometric polynomials). 
This problem  has been widely investigated. For instance, Gine et al.~\cite{GGL13} showed
that a generalized Abel equation
\[
x'=A_0(t)+A_1(t)x+\cdots+A_n(t)x^n
\]
has at most $n$ polynomial solutions when $A_i(t)$ are polynomials. For other related results, see for instance \cite{CGM17,GTZ16,LV18,OV20,V20}
and references therein.

In the present work, we consider~\eqref{eq:Abel} with polynomial coefficients 
and examine the number of rational solutions. This problem has 
been analysed for rational limit cycles, i.e., rational solutions $x(t)$
defined for all $t\in[0,1]$, closed ($x(0)=x(1)$), and such that 
there is no continuum of closed solutions. In this context, Liu et al.~\cite{LLWW} 
proved the existence of Abel equations with at least two rational limit 
cycles. Also, Llibre and Vals~\cite{LLV} (see also~\cite{V20r}) claim that Abel equations have at most two rational 
limit cycles.

In precise terms, the problem is the following. 

%
%
%

Fixed $A$ and $B\in\mathbb{C}[t]$. By definition 	a rational solution of~\eqref{eq:Abel} is a solution $x\colon I \to\mathbb{R}$, where $I$ is an interval on $\mathbb{R}$  such that $x(t)=q(t)/p(t)$, for some $p,q\in\mathbb{C}[t]$, $p\not\in\mathbb{C}$, 	and $p(t)\neq 0$ for any $t\in I$. 
Note that  $q(t)/p(t)$ is also a rational solution on any interval where $p$ has no roots. To avoid such multiplicity, we count all these solutions	as one. 
	
In Proposition~\ref{prop:inv}, it is proved that $p$ divides $A$ and therefore $A(t)\neq 0$ implies  $p(t)\neq 0$. 	
So if we fix an interval $I$ where $A(t)\neq0$ for every $t\in I$,  we may define the number of rational solutions of \eqref{eq:Abel} 	as the number of rational solutions of~\eqref{eq:Abel} defined on $I$.

The main result obtained is Theorem \ref{theo:A}. It
states that if $n:=\deg(A)$ is even or $\deg(B)>(n-1)/2$ then~\eqref{eq:Abel} has at most two rational solutions, and
if $n$ is odd then $\binom{n }{(n+1)/2}+1$ is an upper bound for the number of rational solutions. This upper bound is not sharp, as we shall show in the final section.

The problem of studying the rational solutions of \eqref{eq:Abel} is equivalent to considering the existence of invariant curves of degree one in $x$, i.e., curves of the form $p(t)x+q(t)=0$. Note that $x=0$ is also an invariant curve. Using Darboux's theory of integrability, we study when there 
exist $\alpha_0,\dots,\alpha_r \in \mathbb{C}$ and invariant curves $p_i(t)x+q_i(t)=0$ such that
the function
\begin{equation}\label{eq:Darboux}
f(t,x):=x^{\alpha_0}\prod_{i=1}^r (1+p_i(t)x)^{\alpha_i}
\end{equation}
is a first integral of~\eqref{eq:Abel}.

Gin\'e and Santallusia~\cite{GS10} characterized the Abel equations with a Darboux first integral of the form~\eqref{eq:Darboux} 
when $r=2$ and $r=3$. Here, we shall be interested in the relation between $r$ and $n$. 

Our main result in this setting, Theorem~\ref{theo:C}, states  that if the number 
of rational solutions, $r$, is greater than or equal to $(n+1)/2$ then the equation admits a Darboux first integral of the form~\eqref{eq:Darboux}. 

In the final
section, we shall look into the low degree cases and provide a computational 
method to obtain the maximum number of rational solutions for a fixed degree of $A$. We show that if $n$ is even or $\deg(B)>(n-1)/2$ then 
the upper bound of two rational solutions is sharp. If $n=1,3$, the upper bound of $\binom{n }{(n+1)/2}+1$
rational solutions is also sharp, but for $n=5$ we obtain a lower upper bound. 
Moreover, for $n=1,3,5,7$, we obtain examples of 
$(n-1)/2$ rational solutions for~\eqref{eq:Abel} without first integrals of the form~\eqref{eq:Darboux}, so 
the bound provided by Theorem~\ref{theo:C} could be sharp.

\section{Algebraic invariant curves of degree one in $x$}

\noindent
Consider the Abel equation \eqref{eq:Abel}, with $g(t,x):=A(t)x^3+B(t)x^2$, and denote the associated vector field by
$$\mathcal{X}=\frac{\partial}{\partial t}+g\frac{\partial}{\partial x}.$$

Let $f\in\mathbb{C}[t,x]$. The curve $f(t,x)=0$ is an algebraic invariant curve of \eqref{eq:Abel} if there exists $K\in\mathbb{C}[t,x]$ such that 

$$
\left(\mathcal{X}f\right)(t,x)=\left(\frac{\partial f}{\partial t}+g\frac{\partial f}{\partial x}\right)(t,x)=K(t,x)f(t,x).
$$

The polynomial $K (t,x)$ is called the cofactor of $f(t,x)$. Observe that $f_0(t,x):=x=0$ is always an algebraic invariant curve of \eqref{eq:Abel} with cofactor $K_0(t,x)=A(t)x^2+B(t)x$.

 If $f(t,x)=0$ is invariant and $x(t)$ is a solution of \eqref{eq:Abel} then, for any $t_0$ in the domain of the solution,
$$f(t,x(t))=f(t_0,x(t_0))\exp\left(\int_{t_0}^{t} K(s,x(s))ds\right).$$
	
Therefore, if $f(t_0,x(t_0))=0$, then $f(t,x(t))=0$ for all $t$. Consequently $f(t,x)=0$ consists of trajectories of solutions of the equation. 

Assume that \eqref{eq:Abel} has an invariant algebraic curve of degree one in $x$, i.e., a curve of the form $q(t)+p(t)x=0$ where $p(t)$ and $q(t)$ are complex polynomials. 

Henceforth, we shall refer to such invariant curves simply as invariant curves of \eqref{eq:Abel},
and assume $q(t)\not \equiv 0$, with $p,q$ coprime. As $x(t)\equiv 0$ is a solution of~\eqref{eq:Abel}, 
then for any invariant curve either $q(t)\neq 0$ for every $t\in\mathbb{R}$ or $q(t)\equiv 0$. Moreover, 
in \cite{LLWW} it is proved that $q(t)$ is constant. We include the proof for convenience.

\begin{prop}[\cite{LLWW}, Lemma 2]\label{prop:inv}
	The curve $q(t)+p(t)x=0$ is an invariant curve of equation \eqref{eq:Abel} if and only if $q(t)$ is a constant $c \in\mathbb{C}\setminus\{0\}$ and 
	\begin{equation}\label{Condinv}
		p(t) p'(t) - c\, p(t) B(t) + c^2 A(t)=0.
	\end{equation}
\end{prop}

\begin{proof}
	Take $f(t,x)=q(t)+p(t)x=0$ to be an invariant curve of the differential equation \eqref{eq:Abel}, i.e., there exists a polynomial $K(t,x)$ such that
$$\mathcal{X}f=f_t+f_x g=f K.$$
Taking into account the degree in $x$ on both sides of the above equality, one obtains that $K(t,x)=K_0(t)+K_1(t)x+K_2(t)x^2$, where $K_i\in \mathbb{C}[t]$, $i=0, 1, 2$. Then
$$q'(t)+p'(t) x+p(t)(A(t)x^3+B(t)x^2)=$$ $$=(q(t)+p(t)x)(K_0(t)+K_1(t)x+K_2(t)x^2).$$
Equating the coefficients of the same powers of $x$, one obtains the following:
\begin{align*}
	\begin{split}
		K_2(t)&=A(t),\\
		p(t)B(t)&=p(t)K_1(t)+K_2(t)q(t),\\
		p'(t)&=K_1(t)q(t),\\
		q'(t)&=q(t)K_0(t).
	\end{split}
\end{align*}

From the last of these equalities, one has that either $q(t)\equiv0$ or $K_0(t)\equiv0$. As we are assuming that $q(t)\not\equiv0$, then $K_0(t)\equiv0$. Consequently, $q(t)=c\in\mathbb{C}\setminus\{0\}$. 

From the other equalities, one has
$$p(t)B(t)=p(t)\frac{p'(t)}{c}+cA(t)$$
or equivalently
$$p(t)p'(t)-cp(t)B(t)+c^2A(t)=0.$$ 
\end{proof}

An immediate consequence of \eqref{Condinv} is that $p(t)$ divides $A(t)$. Thus an invariant curve of \eqref{eq:Abel} has the form $c+p(t)x=0$, with $c \in \mathbb{C} \setminus \{0\}$ and $p\in\mathbb{C}[t]$ a divisor of $A(t)$. 


When $p(t)\equiv K$ is a nonzero constant, the Abel equation becomes the separated variable equation $x'=B(t)x^2(K/c+x)$ whose unique rational solutions are $0$ and $-c/K$. Hence, we shall only consider the case $\deg(p) \geq 1$.

Without loss of generality, we assume that $c=1$. Then \eqref{Condinv} becomes
\begin{equation}
	p(t)p'(t)-p(t)B(t)+A(t)=0.\label{CondInv}
\end{equation}

As $p(t)$ divides $A(t)$, then there must exist $r(t) \in\mathbb{C}[t]$ such that $A(t)=p(t) r(t)$, and then \eqref{CondInv} transforms into
\begin{equation}\label{exprB}
	B(t) =  p'(t) + r(t).
\end{equation}

The following result establishes a relation between $\deg(B)$ and the degree of an invariant curve of \eqref{eq:Abel}.

\begin{lema}\label{Lema degs}

Let $A(t)=a_n t^n + a_{n-1} t^{n-1} + \cdots + a_0$  with $a_n \neq 0$, and $p(t)=p_m t^m + p_{m-1} t^{m-1} + \cdots + p_0 $ with $p_m \neq 0$, such that $1+p(t)x=0$ is an invariant curve of \eqref{eq:Abel}. Then  
	\begin{itemize}
		\item[(i)] 	$\deg(B)\leq \max(\deg(p)-1,n-\deg(p))$;
		\item[(ii)]  $\deg(p)=(n+1)/2$ if and only if $\deg(B)\leq (n-1)/2$. In this case, $\deg(A/p) = (n-1)/2$ and
		\begin{align*}
			\deg(B) < (n-1)/2, & \quad \text{ if } { mp_m^2+a_n=0},\\
			\deg(B) = (n-1)/2,  & \quad \text{ if } {mp_m^2+a_n\neq0}.
		\end{align*}
	
	\end{itemize}	
\end{lema}

\begin{proof}
	(i) Since $\deg( p') = \deg(p)-1$ and $\deg(A/p) = n - \deg(p)$, the claim follows directly from \eqref{exprB}.
	
\noindent (ii)  $\deg(p) \neq (n+1)/2$ is equivalent to $\deg(p')>(n-1)/2$ or $(n-\deg(p))>(n-1)/2$, which  is also equivalent to $\deg(B)>(n-1)/2$ by \eqref{exprB}.

 If $\deg(p) = (n+1)/2$, then $\deg(A/p) = n - \deg(p) = (n-1)/2$. Hence $\deg(B)\leq(n-1)/2$ because of \eqref{exprB}, and $\deg(B)=(n-1)/2$ if and only if the coefficients of the leading terms of $ p'(t)$ and $ A(t)/p(t)$ do not cancel out.    
\end{proof}

Let  $1 + p_1(t) x =0$ and $1 + p_2(t) x = 0$ be two different invariant curves of \eqref{eq:Abel}, and write $q(t):=\gcd(p_1(t),p_2(t))$. Then
\begin{equation*}
 p_1(t)=s_1(t)q(t)\quad \text{and}\quad p_2(t)=s_2(t)q(t), 
\end{equation*}
where $s_1,s_2\in\mathbb{C}[t]$ have $\gcd(s_1,s_2)=1$. 
Since  $p_1(t)$ and $p_2(t)$ divide $A(t)$, there exists $s\in\mathbb{C}[t]$ such that
\begin{equation*}
	A(t)=q(t)s_1(t)s_2(t)s(t).
\end{equation*}

The next step is to study the relation between $s_1(t)$ and $s_2(t)$, and the expression and properties of $s(t)$. 

\begin{lema}\label{Lemma23}
Assume that $1+p_1(t)x=0$ and $1+p_2(t)x=0$ are two invariant curves of \eqref{eq:Abel}. Then $p_1(t)$ and $p_2(t)$ are not relatively prime, i.e., $\gcd(p_1(t),p_2(t))\neq 1$. 
\end{lema} 
\begin{proof}
Assume that $\operatorname{gcd}(p_1(t), p_2(t)) = 1$. Then $p_1(t) $ and $p_2(t)$ are relatively prime factors of $A(t)$. So $A(t)=p_1(t)p_2(t)s(t)$, and by \eqref{exprB} one has that \[ p'_1(t)+ p_2(t)s(t)= p'_2(t)+ p_1(t)s(t).\] Therefore for some constant $C$
	$$
	p_1(t) -p_2(t)=C \exp\left(\int s(t)\,dt\right),
	$$
in contradiction with $p_1(t)-p_2(t)$ being a polynomial.	
\end{proof}

Note that, with the above notation, what the lemma says is that the monic polynomial $q(t) = \gcd(p_1(t), p_2(t))$ is never constant.

The following result sets up the relation between any two invariant curves of~\eqref{eq:Abel}
and the induced decomposition of $A(t)$.

\begin{prop}\label{Prop s1s2s}
With the above notation, let $\prod_{i=1}^l q_i(t)^{\delta_i}$ be the decomposition of $q(t)$ into irreducible monic polynomials. If $1+p_1(t)x=0$ and $1+p_2(t)x=0$ are two invariant curves of \eqref{eq:Abel}, then there exist $C \in \mathbb{C} \setminus \{0\}$ and non-negative integers $\gamma_i,\ i = 1, \dots, l,$  such that 
\begin{equation} \label{exprs1s2}
	s_1(t) = s_2(t) + C \prod_{i=1}^l q_i(t)^{\gamma_i}
\end{equation}	
 and
	\begin{equation}\label{exprs}		
s(t) = \operatorname{gcd}(q(t),q'(t))\, \sum_{i=1}^l (\delta_i + \gamma_i)\, q'_i(t)\prod_{j \neq i} q_j(t) ,	\end{equation}
where $\prod_{j \neq i} q_j(t)$ is considered equal to one when $l = 1$.
\end{prop}
	
\begin{proof}
	Since $p_i(t)=q(t)s_i(t)$, $i=1,2$, $\frac{A(t)}{p_1(t)}=s_2(t)s(t)$ and $\frac{A(t)}{p_2(t)}=s_1(t)s(t)$, it follows from \eqref{exprB} that
	\[(q(t)' s_1(t) + q(t) s'_1(t)) +  s_2(t) s(t) =  (q'(t) s_2(t) + q(t) s'_2(t)) +s_1(t) s(t),\] and consequently 
	\begin{equation}\label{ecu2} 
		q(t) (s'_1(t) -  s'_2(t))  =  (s(t) - q'(t)) (s_1(t) - s_2(t)) .
	\end{equation} 
	In particular, $ s_1(t) -   s_2(t)$ divides $q(t) (  s'_1(t) -  s'_2(t))$.
	
	One can write $ s_1(t) -  s_2(t) = \tilde{q}(t) h(t)$ where
\begin{equation*}
	\tilde{q}(t)=\prod_{i=1}^l q_i(t)^{\gamma_i},
	\end{equation*}	
for some non-negative integers $\gamma_i,\ i = 1, \ldots, l$, and such that none of the irreducible factors of $h(t)$ divides $q(t)$. So, from \eqref{ecu2} one has that $h(t)$ divides $s_1'(t)-s_2'(t)$. 
	
Moreover, since $  s'_1(t) -  s'_2(t) = \tilde{q}'(t) h(t) + \tilde{q}(t) h'(t)$, one has that $h(t)$ divides $\tilde{q}(t) h'(t)$. 

Now, as none of the irreducible factors of $h(t)$ divide $q(t)$ and $\tilde{q}(t)$ is a product of some (possibly repeated) factors of $q(t)$, the conclusion is that $h(t)$ divides $h'(t)$, which is only possible if $h(t) = C \in \mathbb{C}$. Then \[s_1(t) =  s_2(t) + C \prod_{i=1}^l q_i(t)^{\gamma_i}\] for some non-negative integers $\gamma_i,\, i = 1, \dots, l$. Furthermore, $C \neq 0$ because $ p_1(t) \neq  p_2(t)$ by hypothesis.
	
	Substituting in \eqref{ecu2} and solving for $s(t)$, one obtains that
	\begin{align*}
		 s(t) & = q'(t) + \prod_{i=1}^l q_i(t)^{\delta_i} \frac{ C \sum_{i=1}^l \gamma_i\, q'_i(t)\, q_i(t)^{\gamma_i-1} \prod_{j \neq i} q_j(t)^{\gamma_j}}{C \prod_{i=1}^l q_i(t)^{\gamma_i}}\\ & = q'(t) + \prod_{i=1}^l q_i(t)^{\delta_i - 1} \sum_{i=1}^l \gamma_i\, q'_i(t) \prod_{j \neq i} q_j(t)\\ & =  q'(t) + \operatorname{gcd}(q(t),q'(t)) \sum_{i=1}^l \gamma_i\, q'_i(t) \prod_{j \neq i} q_j(t).
\end{align*}
Finally, taking into account that \begin{align*}
q'(t) & = \sum_{i=1}^l \delta_i\, q'_i(t)\, q_i(t)^{\delta_i-1} \prod_{j \neq i} q_j(t)^{\delta_j} \\ & =  \operatorname{gcd}(q(t),q'(t)) \sum_{i=1}^l \delta_i\, q'_i(t) \prod_{j \neq i} q_j(t),
\end{align*}	
one can conclude that \[s(t) =  \operatorname{gcd}(q(t),q'(t))\, \sum_{i=1}^l (\delta_i + \gamma_i)\, q'_i(t) \prod_{j \neq i} q_j(t).\] 
\end{proof}

The converse of Proposition~\ref{Prop s1s2s} also holds. Therefore, it parametrizes
all cases of equation~\eqref{eq:Abel} having at least two invariant curves.

\begin{prop}
Given $s_2(t)\in\mathbb{C}[t],\ C \in \mathbb{C} \setminus \{0\}$, non-negative integers $\gamma_i,\ i = 1, \dots, l$, and irreducible monic polynomials $q_i(t),\ i = 1, \dots, l$, such that $\gcd(s_2(t), q_i(t)) = 1$ if $\gamma_i \neq 0$. Set $q(t) := \prod_{i=1}^l q_i(t)^{\delta_i}$, for some positive integers $\delta_i ,\ i = 1, \dots, l$, and define $s_1(t)$ and $s(t)$ by \eqref{exprs1s2} and \eqref{exprs}, respectively. If
	\[A(t)=q(t)s_1(t)s_2(t)s(t),\quad B(t)=(q(t)s_1(t))'+s_2(t)s(t),\]
	\[p_1(t)=q(t)s_1(t),\quad p_2(t)=q(t)s_2(t)\]
then $1+p_1(t)x=0$ and $1+p_2(t)x=0$ are two invariant curves of \eqref{eq:Abel}. 
\end{prop}

\begin{proof}
The polynomial $p_1(t)$ satisfies \eqref{CondInv} by definition and, by construction, $\gcd(p_1(t), p_2(t)) = q(t)$. Now, it suffices to note that \eqref{exprs1s2} and \eqref{exprs} imply \eqref{ecu2}, from which one has that $p_2(t)$ satisfies \eqref{CondInv}. 
\end{proof}

\begin{rema}\label{Rem 2sol}
In particular, considering $q(t)=q_1(t)=t$, $s_2(t)=1-t^{n-1}$, $C=1$,
and $\gamma_1=n-1$, one obtains that for every $n$ there exist $A(t)$ and $B(t)$ with $\deg(A)=n$ such that~\eqref{eq:Abel} has at least two invariant curves.
\end{rema}

Once the expression of $s(t)$ is known, it is possible to establish relations between the degrees of the polynomials $p_1(t)$ and $p_2(t)$ that define the invariant curves, and, as a consequence, to provide upper bounds for the number of invariant curves.

\begin{prop}\label{Cor degs}
Let $n = \deg(A)$. Under the hypotheses in Proposition \ref{Prop s1s2s}, the following statements hold:
	\begin{itemize}
		\item[(i)] $\deg(s) = \deg(q)-1$,
		\item[(ii)]  $n+1 = \deg(p_1)+\deg(p_2)$,
		\item[(iii)] $\deg(p_1)=\deg(p_2)$ if and only if $\deg(p_1)=(n+1)/2$.
	\end{itemize}
\end{prop}

\begin{proof}
	(i) By \eqref{exprs}, $\deg(s)\leq\deg(q)-1$. Now, noting that in the expression of $s(t)$ both $\delta_i$ and $\gamma_i$ are
nonnegative integers, the equality is obtained.  
	
	(ii) As $A(t)=q(t)s_1(t)s_2(t) s(t) $, $p_1(t)=q (t)s_1(t)$ and $p_2(t)=q(t)s_2(t)$, it follows that
	\begin{equation*}\begin{split}
			n = \deg(A) = & \deg(q)+(\deg(p_1)-\deg(q)) +\\ & +  (\deg(p_2)-\deg(q))+\deg(s)\\
			= & \deg(p_1)+\deg(p_2)-\deg(q)+\deg(s)\\
             = & 
			 \deg(p_1)+\deg(p_2)-1,
		\end{split}
	\end{equation*}
	from which (ii) stands.
	
	(iii) This is a direct consequence of (ii).
\end{proof}

\begin{coro}\label{3curv}
	If equation \eqref{eq:Abel} has three invariant curves, then they are all of degree $(n+1)/2.$
\end{coro}

\begin{proof}
	This is a direct consequence of (ii) and (iii) of Proposition \ref{Cor degs}.
\end{proof}

Next we study the number of invariant curves whose polynomial coefficients of $x$ are proportional.

\begin{prop}\label{prop:proportional}
Let $K \in \mathbb{C}$, $K\neq0,1$. The curves $1+p(t)x=0$ and $1+Kp(t)x=0$ are invariant curves of \eqref{eq:Abel} if and only if
\[A(t) = K p(t) p'(t),\quad B(t) = (K+1) p'(t).\]
\end{prop}

\begin{proof} By~\eqref{exprB}, the curve $1+p(t)x=0$ is invariant if and 
only if there exists $r(t) \in\mathbb{C}[t]$ such that  $A(t)=p(t)r(t)$, $B(t)=p'(t)+r(t)$.
	
	If $1+Kp(t)x=0$ is invariant then there exists $\overline{r}(t) \in\mathbb{C}[t]$ such that
	\[
	A(t)=Kp(t)\bar{r}(t)=p(t)r(t), \quad B(t)=Kp'(t)+\bar{r}(t)=p'(t)+r(t).
	\]
	Hence $r(t)=Kp'(t)$, and therefore
	\[A(t)=Kp(t)p'(t),\quad B(t)=(1+K)p'(t).\]
	
	Conversely, if $A(t) = K p(t) p'(t)$ and $B(t) = (K+1)p'(t)$ then, taking $r(t)=K p'(t)$, one has
	$ A(t)=p(t)r(t)$, $B(t)=p'(t)+r(t)$. Consequently $1+p(t)x=0$ is invariant. With a similar argument, 
	one has that $1+Kp(t)=0$ is also invariant. 
\end{proof}

It is said that \eqref{eq:Abel} has a centre if every solution
close to $x(t)=0$ satisfies $x(0)=x(1)$. 

It is said that~\eqref{eq:Abel} satisfies the composition condition if 
there exists $p(t)\in\mathbb{C}[t]$ such that $A(t)=a(p(t))p'(t)$, $B(t)=b(p(t))p'(t)$, and $p(0)=p(1)$.

If \eqref{eq:Abel} satisfies the composition condition then, by the change of variable $x(t)=X(p(t))$,
it has a centre. This centre is called a composition centre.

\begin{rema}
Let $A(t) = K p(t) p'(t)$ and $B(t) = (K+1) p'(t)$ for some constant $K$. 
\begin{itemize}
\item[a)] If $K=-1$,  \eqref{eq:Abel} reduces to $x'=p(t)p'(t)x^3$ with the rational first integral $\frac{1}{x^2}+p^2(t)$.
\item[b)] The change of variable $x(t)=X(p(t))$ transforms \eqref{eq:Abel} into
$$
X'(s)=sKX^3(s)+(K+1)X^2(s), \quad s=p(t),
$$
which, with the change $Y(s)=s X(s)$, becomes the separated variable equation
$$
sY'(s)= KY^3(s)+(K+1)Y^2(s)+Y(s).
$$
\item[c)] Note that if $p(0)=p(1)$ then \eqref{eq:Abel} is a composition centre.
\end{itemize}
\end{rema}

If $x_1(t),x_2(t),x_3(t)$ are solutions of~\eqref{eq:Abel}, we say that they are collinear 
if there exist $\alpha_1,\alpha_2,\alpha_3 \in \mathbb{C}$ such that 
\[
\sum_{i=1}^3 \alpha_i=0,\quad \sum_{i=1}^3 \alpha_i x_i =0.
\]

As $x(t)=0$ is always a solution, then \eqref{eq:Abel} has three collinear rational solutions
if and only if these solutions are $x(t)=0$, $x(t)=-1/p(t)$ and $x(t)=-1/(K p(t))$, for some $K \in \mathbb{C}$, $K\neq0,1$. 

Theorem 5 of \cite{GS10} characterizes when the Abel equation
\[x'=A(t)x^3+B(t)x^2+C(t)x+D(t)\]
has three collinear solutions. Therefore, Proposition~\ref{prop:proportional}
solves the equivalent problem restricted to \eqref{eq:Abel}.

\begin{coro}\label{Cor 2curvas}
	Equation \eqref{eq:Abel} has at most two invariant curves whose polynomial coefficients of $x$ are proportional.
\end{coro}

\begin{proof}
	Assume equation \eqref{eq:Abel} has two invariant curves $1+p(t)x=0$ and $1+Kp(t)x=0$, with $K\in\mathbb{C}\setminus\{1\}$. Then, $A(t)=Kp(t) p'(t)$ and $B(t)=(K+1)p'(t)$ by Proposition 2.9. As a direct consequence, there is no $K_1\not\in\{1,K\}$ such that $1+K_1p(t)x=0$.
	
	Furthermore, if $K\neq-1$, from $A(t)=Kp(t) p'(t)$ and $B(t)=(K+1)p'(t)$, 
	$$p(t)=\frac{K+1}{K}\frac{A(t)}{B(t)}.$$ 
		Therefore, if $1+\bar{p}(t)x=0$, $1+\bar{K}\bar{p}(t)x=0$ are invariant then 
	
	$$\bar{p}(t)=\frac{\bar{K}+1}{\bar{K}}\frac{A(t)}{B(t)}=\left(\frac{\bar{K}}{\bar{K}+1}\frac{K+1}{K}\right)p(t),$$ 
	
	and consequently, $\bar{p}(t)=p(t)$ or $\bar{p}(t)=Kp(t)$. 
	
	If $K=-1$, it is enough to recall Remark 2.10 a).
\end{proof}

The following result is the main one of the section. It provides an upper bound to the number of invariant curves and sufficient conditions for \eqref{eq:Abel} to have at most two invariant curves.

\begin{theo}\label{theo:A}
Consider the differential equation \eqref{eq:Abel}, and let $n=\deg(A)$. 
Then there are at most two invariant curves of \eqref{eq:Abel} if one of the following conditions holds:
\begin{itemize}
	\item [(i)] $n$ is even,
	\item [(ii)] $\deg(B)>(n-1)/2$.
\end{itemize}
In any other case, $\binom{n }{(n+1)/2}+1$ is an upper bound for the number of invariant curves of \eqref{eq:Abel}.
\end{theo}

\begin{proof}
	By Corollary \ref{3curv}, there can be three or more different invariant curves if and only if all their respective polynomial coefficients of $x$ have the same degree, $(n+1)/2$. 
	
	(i) If $n$ is even, no invariant curve can have polynomial coefficient of $x$ with degree $(n+1)/2$, so there can be at most two invariant curves.
	
	(ii) If $\deg(B)>(n-1)/2$, by Lemma \ref{Lema degs} (ii), the degree of the polynomial coefficient of $x$ of any invariant curve cannot be $(n+1)/2$, and consequently the equation has at most two invariant curves. 
	
 If $1+p(t)x=0$ is an invariant curve of \eqref{eq:Abel}, then $p(t)$ must divide $A(t)$. Therefore, the candidates for $p(t)$ are the possible divisors of $A(t)$; evidently there are $\binom{n }{(n+1)/2}$ choices for $p(t)$ up to proportionality. Now, by Corollary \ref{Cor 2curvas}, given $p(t)$ there is at most one invariant curve with a proportional polynomial coefficient of $x$. Since in this case $p(t)$ is uniquely determined in terms of $A(t)$ and $B(t)$, our claim follows.
\end{proof}

\begin{rema}
On the one hand, in Remark \ref{Rem 2sol} we showed the existence of equations \eqref{eq:Abel} with at least two invariant curves for every $n$. So, if $n$ is even then  these equations have exactly two invariant curves.

On the other hand, if one takes $q(t) = t$, $s_2(t) = 1-t^{m-1}-t^{m+1}$, $C = 1$, and $\gamma_1 = m+1$ in Proposition \ref{Prop s1s2s}, one  obtains that for every $m \geq 2$ there exist $A(t), B(t)$ with $\deg(A) = 2m + 1$ and $\deg(B) = m+1$ such that \eqref{eq:Abel} has exactly two invariant curves.
\end{rema}

\begin{rema}
If $A(t)$ has multiple roots then the bound given in Theorem~\ref{theo:A} can be sharpened. 

Assume 
that $A(t)=a_1(t)^{\alpha_1}\cdots a_l(t)^{\alpha_l}$, with $\alpha_i\in\mathbb{N}$, $\alpha_1+\cdots+\alpha_l=n$, $0\leq\alpha_i$,
and that \eqref{eq:Abel} has three or more invariant curves.
	
If $1+p(t)x=0$ is an invariant curve of \eqref{eq:Abel}, with $p(t)=a_1(t)^{\beta_1}\cdots a_l(t)^{\beta_l}$, then 
$(\beta_1,\dots \beta_l)$ satisfies
\[0\leq\beta_i\leq\alpha_i,\ i=1,\dots,l,\quad \beta_1+\cdots+\beta_l=(n+1)/2.\] 
The number of different $(\beta_1,\dots,\beta_l)$ satisfying 
the above equations is the number of partitions of $(n+1)/2$ satisfying
the constraints, and can be determined by using generating functions (see \cite{Andrews} for a general approach).

\end{rema}

\section{Darboux integrability and invariant curves}

\noindent
Using Darboux's theory of integrability, we study the maximum number of invariant curves of degree one in $x$ that \eqref{eq:Abel} can have without forcing the existence of a Darboux first integral.

We say that $f(t,x)$, smooth enough and not identically constant, is a first integral of \eqref{eq:Abel} if $\mathcal{X}f=0$, or, equivalently, if $f(t,x(t))$ is constant, when $x(t)$ is a solution of the equation. 

We say that a first integral $f$ of \eqref{eq:Abel} is a Darboux first integral of \eqref{eq:Abel} if $f(t,x)=\prod_{i=1}^{r}f_i(t,x)^{\alpha_i}$, where $f_i(t,x)=0$ are invariant curves of the equation and $\alpha_i\in\mathbb{C}$.

First, we present the general Darboux result which links the existence of a Darboux first integral with the linear dependence of the cofactors of the invariant curves. We have adapted its statement to the present setting.

\begin{theo}[Darboux's Theorem, \cite{Darboux}]\label{IntegralDarboux}
	Let $f_0(t,x)=0, \dots, f_r(t,x)=0$ be invariant curves of \eqref{eq:Abel} with  cofactors $K_0(t,x), \dots, K_r(t,x)$, respectively. If there exist $\alpha_0, \dots, \alpha_r\in\mathbb{C}$ such that $\sum_{i=0}^{r}\alpha_i K_i(t,x)=0$, then $f(t,x)=\prod_{i=0}^{r}f_i(t,x)^{\alpha_i}$ is a first integral of \eqref{eq:Abel}.
\end{theo}

The following result is a direct application of Theorem \ref{IntegralDarboux} when the invariant curves of equation \eqref{eq:Abel} are of degree one in $x$. See \cite{GS10} for a more general version.

\begin{prop}\label{prop:Darboux} Assume that \eqref{eq:Abel} has the invariant curves $1+p_i(t)x=0$, $i=1,\dots,r$.  Let $\alpha_i\in \mathbb{C}$, $i=1,\dots,r$, and $\alpha_0=-\sum_i^r \alpha_i$. Then $f(t,x)=x^{\alpha_0}\prod_{i=1}^r (1+p_i(t)x)^{\alpha_i}$ is a first integral of \eqref{eq:Abel} if and only if 
	\begin{equation}\label{eq:alphacondition}
		\sum_{i=1}^{r}\alpha_i  \frac{A(t)}{p_i(t)}=0.
	\end{equation}
\end{prop} 
\begin{proof}
	If $1+p_i(t)x=0$ is an invariant curve of \eqref{eq:Abel} then its cofactor is $K_i(t,x)=A(t)x^2+ p_i'(t)x$. Recall that $p_i(t)$ divides $A(t)$ by \eqref{Condinv} and  that  $f_0(t,x)=x=0$ is always  an invariant curve with cofactor $K_0(t,x)=A(t)x^2+B(t)x$. 
	
	Now assume the existence of $\alpha_i$, $i=1,\dots,r$ such that \eqref{eq:alphacondition} holds.
	Let $\alpha_0=-\sum_{i=1}^{r}\alpha_i$. By Proposition 3.1, if  $\sum_{i=0}^{r}\alpha_i K_i(t,x)=0$ then \eqref{eq:Abel} is Darboux integrable. Thus
	\begin{align*}
		\sum_{i=0}^{r}\alpha_i K_i(t,x)&=\alpha_0 K_0(t,x)+\sum_{i=1}^{r}\alpha_i K_i(t,x)\\ &=-\Big(\sum_{i=1}^{r}\alpha_i\Big)(A(t)x^2+B(t)x)+\Big(\sum_{i=1}^{r}\alpha_i\Big)(A(t)x^2+p_i'(t)x)\\
		&=\Big(\sum_{i=1}^{r}\alpha_i (p_i'(t)-B(t))\Big)x,
	\end{align*}
	but, as $p_i'(t)-B(t)=  A(t)/p_i(t)$ because of \eqref{Condinv}, one then has
	$$\Big(\sum_{i=1}^{r}\alpha_i( p_i'(t)-B(t))\Big)x=\Big(\sum_{i=1}^{r}\alpha_i  \frac{A(t)}{p_i(t)}\Big)x=0.$$
	Conversely, if $f(t,x)=x^{\alpha_0}\prod_{i=1}^r (c_i+p_i(t)x)^{\alpha_i}$ is a first integral then, from
	$$
	(\mathcal{X}f)(t,x)=\left(\alpha_0A(t)x^2+\alpha_0B(t)x +\sum_{i=1}^r \alpha_i(A(t)x^2+p'_i(t)x)\right)f(t,x)=0,
	$$
	\eqref{eq:alphacondition} follows.
\end{proof}

Now we can state the main integrability result, which provides, in terms of $\deg(A)$, the maximum number of invariant curves of degree one in $x$ that equation \eqref{eq:Abel} can have without having a Darboux first integral. Note that, as seen in the previous result, the cofactors are linearly dependent if and only if there exist $\alpha_0,\dots,\alpha_r\in\mathbb{C}$ such that
$$\sum_{i=0}^{r}\alpha_i \frac{A(t)}{p_i(t)}=0.$$

\begin{theo}\label{theo:C}
	Let $n \geq 3$. If \eqref{eq:Abel} has more than $(n+1)/2$ invariant curves of degree one in $x$ then \eqref{eq:Abel} has a Darboux first integral.
\end{theo}

\begin{proof}
	Assume that $1+p_1(t)x=0,\dots,1+p_l(t)x=0$ are invariant curves, with $l>(n+1)/2$. By Corollary~\ref{Cor degs}, $\deg (p_i)=(n+1)/2$ for $i=1,\dots,l$.
	
Since $\deg(A/p_i)=n-(n+1)/2=(n-1)/2$,  $A/p_i\in\mathbb{C}_{(n-1)/2}[t]$ for  $i=1,\dots,l$, and $l>(n+1)/2=\dim(\mathbb{C}_{(n-1)/2}[t])$, there must exist $\alpha_i\in\mathbb{C},\ i = 1, \dots, l,$ such that
	$$\sum_{i=1}^{l}\alpha_i\frac{A(t)}{p_i(t)}=0.$$
	By Proposition~\ref{prop:Darboux}, equation \eqref{eq:Abel} has a Darboux first integral. 
\end{proof}

\section{Low-degree equations}

\noindent
In this section, we shall present examples of polynomials $A(t)$ and $B(t)$ such that equation \eqref{eq:Abel} has either two or more than two 
invariant curves, where, as in the previous sections, by invariant curve we mean invariant curve of degree one in $x$. Specifically, we shall determine all the invariant curves for $\deg(A)=1$ and $\deg(A)=3$, and describe a computational method to obtain all invariant curves, particularizing it to $\deg(A) = 5$ to illustrate the effective feasibility of the computations.

We say two invariant curves, $1+p_1(t)x=0$, $1+p_2(t)x=0$, $p_1,p_2\in\mathbb{C}[t]$, are linearly independent if the corresponding rational solutions $-1/p_1(t)$, $-1/p_2(t)$ are so.

In the simplest case, $\deg(A)=1$, we obtain that there are at most two invariant curves.

\begin{prop}
	If $\deg(A) =1$, then equation \eqref{eq:Abel} has two different invariant curves if and only if there exist $p(t) \in \mathbb{C}[t]$ and $K \in \mathbb{C} \setminus \{0,1\}$ such that $A(t) = K p(t) p'(t)$ and $B(t) = (K+1) p'(t)$.
\end{prop}

\begin{proof}
Assume that $1+p_1(t) x = 0$ and $1+p_2(t) x = 0$ are two different invariant curves of \eqref{eq:Abel}.  As $p_1(t),p_2(t)\not\in\mathbb{C}$, divide $A(t)$ and $\deg(A)=1$, one has that $p_1(t)$ and $p_2(t)$ are proportional. By Proposition~\ref{prop:proportional}, the conclusion follows. The converse can be checked directly.
\end{proof}

Consider now $\deg(A)=3$. By Theorem \ref{theo:A}, the maximum number of invariant curves is four.  By Proposition~\ref{Prop s1s2s}, if  \eqref{eq:Abel} has two different invariant curves then there exist $q(t),s_1(t),s_2(t)\in\mathbb{C}[t]$, $q(t)$ monic and not constant (see Lemma \ref{Lemma23} and the comment that follows it) such that the invariant curves are $1 + p_1(t) x = 0$ and $1 + p_2(t) x = 0$, where $p_1(t)=q(t) s_1(t)$ and $p_2(t)=q(t) s_2(t)$. Assume that \eqref{eq:Abel} has more than two different invariant curves. By Theorem \ref{theo:A}, one assumes that $\deg(B)\leq 1$. By Corollary~\ref{Cor 2curvas}, two of the invariant curves are not proportional, so one may assume that $p_1(t),p_2(t)$ are not proportional. Moreover, by Corollary~\ref{3curv}, $\deg(p_1)=\deg(p_2)=2$. In particular, $\deg(q)=1$. Finally, by Proposition~\ref{Prop s1s2s}, either $s_1(t) - s_2(t) = C$ or $s_1(t) - s_2(t) = C\, q(t)$. 

We are now in condition to characterize all families having more than two invariant curves for $\deg(A)=3$. 

\begin{prop}
Let $\deg(A) =3$. Let $1+p_1(t)x=0$ and $1+p_2(t)x=0$ be two invariant curves of \eqref{eq:Abel} where $p_1(t)$ and $p_2(t)$ are non-proportional and $\deg(p_1)=\deg(p_2)=2$. Let $c_i$ be  the coefficient of the leading term of $p_i(t)$, $i=1,2$. 

	\begin{enumerate}	
		\item If $s_1(t) - s_2(t) = C$ then \eqref{eq:Abel} has at least three invariant curves. Moreover, 
		it has four invariant curves if and only if there exists $k \in \{-1,1/2,2\}$ such that $C = k( c_1 q(0) - s_1(0))$. 
		
		\item If $s_1(t) - s_2(t) = C\, q(t)$ then there exists a third invariant curve if and only if $C = -c_1$ or $C = c_1/2$. Moreover, in these cases there are exactly four invariant curves. 
	\end{enumerate}
\end{prop}

\begin{proof}
	Assume that $s_1(t)-s_2(t)=C$. Since $p_1(t)$ and $p_2(t)$ are non-proportional, one has that $\deg(s_1)=\deg(s_2)=\deg(q)=1$ and $\deg(s)=0$ by Proposition \ref{Cor degs}. From \eqref{exprs1s2} and \eqref{exprs}, $s=1$.  Then 
	\[A(t) = q(t) s_1(t) (s_1(t) -C),\quad B(t) = (q(t) s_1(t))' + s_1(t) - C.\]
	Now it is straightforward to check that  $1 + p_3(t) x = 0,$ with 
	\[p_3(t) = \frac{1}{c_1}\, s_1(t) (s_1(t) - C),\] is an invariant curve.  
	
	Since in this case $A(t) = q(t) s_1(t) (s_1(t)-C)$, any other invariant curve must correspond to a multiple, $K$, of $p_1(t), p_2(t)$, or $p_3(t)$. Now, using Proposition \ref{prop:proportional} and condition \eqref{CondInv} in each case, one obtains that $K = 1/2$ and that $C = k( c_1 q(t) - s_1(t)),$ with $k = 1/2,\ k = -1$, or $k = 2$, respectively. Note that, as $q(t),s_1(t)$ are degree one polynomials and $q(t)$ is monic, this expression does not depend on $t$.
	
Assume now that  $s_1(t) - s_2(t) = C\, q(t)$.  Then $\deg(s_1)=\deg(s_2)=\deg(q)=1$ and $\deg(s)=0$. But now one has $s=2$. Therefore
	 \[A(t) = 2 q(t) s_1(t) (s_1(t) - C q(t))\] and \[B(t) = (q(t) s_1(t))' + 2(s_1(t) - C q(t)) = 3 s_1(t) + (c_1 - 2 C) q(t).\] Since the leading terms of $p_1(t)$ and $p_2(t)$ are distinct and $p_3(t)$ must divide $A(t)$, the only possibilities for a third invariant curve $1 + p(t) x = 0$ are 
	\[p(t) = \frac{c_1 - C}{c_1} p_1(t), \quad p(t) = \frac{c_1}{c_1 - C} p_2(t),\] 
	\[p(t) = \frac{1}{c_1} s_1(t) (s_1(t) - C q(t)),\quad p(t) = \frac{1}{c_1 - C} s_1(t) (s_1(t) - C q(t)).\] 
	
	Now, using \eqref{CondInv}, one obtains that $C = -c_1$ or $C = c_1/2$, and, by direct computation, that $1 + p_3(t) x = 0$ and $1 + p_4(t) x = 0$ are invariant curves, where
	\begin{itemize}
		\item if $C = -c_1$ then  $p_3(t) = (1/c_1)\, s_1(t) (s_1(t) - q(t))$ and $p_4(t) = 2\, q(t) s_1(t)$.
		\item 
		if $C = c_1/2$ then  $p_3(t) = (2/c_1)\, s_1(t) (s_1(t) -  q(t))$ and $p_4(t) = 2\, q(t) (s_1(t) - q(t))$.
	\end{itemize}
	
	The converse is just a direct checking.
\end{proof}

\begin{rema}
Observe that for $\deg(A) = 1$ and $\deg(A) = 3$ we can find suitable $A(t)$ and $B(t)$ such that the upper bound given in Theorem \ref{theo:A} is attained.
\end{rema}

In order to study the case $\deg(A) = 5,$ we followed a different strategy which can theoretically be extended to polynomials $A(t)$ of higher degrees. This strategy is summarized in the following remark. 

\begin{rema}\label{RemTech}
Given $A(t)$ and $B(t)$. Let $n = \deg(A)$ and let $p(t)$ be a generic polynomial of degree $(n+1)/2$. Then $1+p(t)x=0$ is an invariant curve of \eqref{eq:Abel} if and only if \eqref{CondInv} holds, i.e.
 \[
 p(t) p'(t) + A(t) - B(t) p(t) = 0.
 \]
  Therefore, by considering the coefficients of $p(t)$ as indeterminates, $x_0, \dots,$ $ x_{(n+1)/2},$ the coefficients of the powers of $t$ in $p(t) p'(t) + A(t) - B(t) p(t)$ form an ideal $I = \langle f_0, \ldots, f_n \rangle$ of $\mathbb{C}[x_0, \dots,$ $ x_{(n+1)/2}]$. Now, we can take advantage of computational commutative algebra techniques to look for possible solutions of the polynomial system $f_0 = \cdots = f_n = 0$. This can be done, for instance, using the \texttt{minAss} command in Singular \cite{Singular}. For the reader interested in delving into this technique, we would refer them to \cite[Chapter 1]{RS} or \cite[Section 3]{BFOS}.
\end{rema}

\begin{exam}
Let \[A(t) = 4(t-1)t(t+1)(3t-1)(3t+2)\quad \text{and}\quad B(t) = 6(4t^2+t-1),\] and consider $p(t) = x_3 t^3 + x_2 t^2 + x_1 t + x_0$. In this case, the coefficients in $t$ of  $p(t) p'(t) + A(t) - B(t) p(t)$ are:
\begin{align*}
f_0 & = x_1x_0+6x_0,\\
f_1 & = 2x_2x_0+x_1^2+6x_1-6x_0+8,\\
f_2 & = 3x_3x_0+3x_2x_1+6x_2-6x_1-24x_0-12,\\
f_3 & = 4x_3x_1+6x_3+2x_2^2-6x_2-24x_1-44,\\
f_4 & = 5x_3x_2-6x_3-24x_2+12,\\
f_5 & = 3x_3^2-24x_3+36.
\end{align*}
Using the \texttt{minAss} command in Singular to compute the
minimal associated primes of the ideal $I = \langle f_0,
\dots, f_5\rangle \subset \mathbb{C}[x_0,x_1,x_2,x_3]$,
we obtain
\[\begin{array}{l}
x_0+4 = x_1+6 = x_2-4 = x_3-6 = 0, \\
x_0 = x_1+2 = x_2-4 = x_3-6 = 0, \\
x_0 = x_1+2 = x_2 = x_3-2 = 0.
\end{array}\]
Hence, we conclude that in this example \eqref{eq:Abel} has exactly three invariant curves. It is easy to check that they are linearly independent.
\end{exam}

Now, let us thoroughly analyse the case $\deg(A)=5$.

If $\deg(A) = 5$ and we assume that \eqref{eq:Abel} has more than
two invariant curves then $p_1(t)$ and $p_2(t)$ have degree
$3$, the monic polynomial $q(t)$ has degree $3, 2$, or $1$,
and consequently $s_i(t),\ i = 1,2$, has degree $0, 1$, or
$2$, respectively. If $q(t)$ has degree $3$, then $p_1(t)$
and $p_2(t)$ are proportional in contradiction with our
prior assumption. Moreover, if $q(t)$ has degree 2, by a
linear change of the independent variable in
\eqref{eq:Abel} one may assume that $q(t)$ is equal to
$t^2-1$ or to $t^2$ (according to whether it has simple or multiple
roots). Analogously, if $q(t)$ has degree $1$, we change the
independent variable in \eqref{eq:Abel} in such a way that
$s_1(t)$ becomes equal to $c_1 (t^2-1)$ or to $c_1 t^2$.
These changes of variable help to reduce the number of
parameters, favouring the computational tasks mentioned in
Remark \ref{RemTech} without impact on the number of
invariant curves of \eqref{eq:Abel}.

Proceeding in this way, we have to distinguish the following cases:
\begin{itemize}
\item[1.] $q(t) = t^2-1$ and $s_1(t) = c_1 (t - z)$, with $c_1 \in \mathbb{C} \setminus \{0\}$; and consequently, 
\begin{itemize}
\item[a.] $s_2(t) = s_1(t) - C$ and $s(t) = q'(t),$ with $C \in \mathbb{C} \setminus \{0\}$;
\item[b.] $s_2(t) = s_1(t) - C (t \pm 1)$ and $s(t) = q'(t)+(t \mp 1),$ with  $C \in \mathbb{C} \setminus \{0,c_1\}$.
\end{itemize}
\item[2.] $q(t) = t^2$ and $s_1(t) = c_1 (t - z)$, with $c_1 \in \mathbb{C} \setminus \{0\}$; and consequently, 
\begin{itemize}
\item[a.] $s_2(t) = s_1(t) - C$ and $s(t) = q'(t),$ with $C \in \mathbb{C} \setminus \{0\}$;
\item[b.] $s_2(t) = s_1(t) - C t$ and $s(t) = q'(t)+t,$ with $C \in \mathbb{C} \setminus \{0,c_1\}$.
\end{itemize}
\item[3.] $q(t) = t-z$ and either $s_1(t) = c_1 (t^2-1)$ or $s_1(t) = c_1 t^2$, with $c_1 \in \mathbb{C} \setminus \{0\}$; and consequently, 
\begin{itemize}
\item[a.] $s_2(t) = s_1(t) - C$ and $s(t) = q'(t),$ with $C \in \mathbb{C} \setminus \{0\}$;
\item[b.] $s_2(t) = s_1(t) - C (t-z)$ and $s(t) = q'(t)+1,$ with $C \in \mathbb{C} \setminus \{0\}$;
\item[c.] $s_2(t) = s_1(t) - C (t-z)^2$ and $s(t) = q'(t)+2,$ with $C \in \mathbb{C} \setminus \{0,c_1\}$.
\end{itemize}
\end{itemize}

Now it suffices to apply the method described in Remark
\ref{RemTech} to characterize all cases of equation
\eqref{eq:Abel} having more than two invariant curves for
$\deg(A)=5$.

\begin{prop}
If $\deg(A)=5$ then \eqref{eq:Abel} has at most five invariant curves.
\end{prop}

\begin{proof}
The proof follows by direct computation.
\end{proof}

In order to facilitate understanding the computational tasks, let us analyse one of the cases in some detail.

\begin{exam}
In Case 1.b, one has that 
\[A(t) = -c_1 (t-1) (t+1) (3t+1) (t-z) \left((C-c_1) t + c_1 z-C\right)\]
and 
\[B(t) = (6 c_1-3 C) t^2 + (c_1+2 C-5 c_1 z) t -c_1 z-c_1+C.\]
Set $p(t) = x_3 t^3 + x_2 t^2 + x_1 t + x_0$. The ideal $I \subset \mathbb{C}[c_1,C,z,x_3,x_2,x_1,x_0]$ of the coefficients in $t$ of $p(t) p'(t) + A(t) - B(t) p(t)$ decomposes as an intersection of $24$ ideals (obtained with \texttt{minAss} and called minimal primes of $I$). One of these ideals, for example, is that generated by $3C-c_1,3z-1, 3x_0+c_1, x_1+c_1, 3 x_2-c_1,$ and $x_3 - c_1$. This means that $-1/(c_1 t^3 + c_1/3 t^2 -c_1 t - c_1/3)$ is an invariant curve of \eqref{eq:Abel} subject to the constraint $C = c_1/3$ and $z = 1/3$.

Given a component of $I$, one can obtain the corresponding constraints for the parameters $c_1, C,$ and $z$ by eliminating the variables $x_i,\ i = 0, \dots, 3$. This can be performed in Singular with the command \texttt{eliminate}.

Combining the compatible constraints given by each component of $I$, one obtains that
\begin{itemize}
\item $C=c_1/3$ and $z \in \{-5/3,-7/9,0,1/3\}$;
\item $C=-2 c_1$ and $z \in \{0,1/3,3\}$;
\item $C=-c_1/2$ and $z \in \{-3,-1/2,-5/3,0\}$;
\item $C=2/3 c_1$ and $z \in \{5/3,7/9,2/3\}$.
\end{itemize}
Now, one can easily check that, for example, if $C = c_1/3$ and $z = 1/3$ then \eqref{eq:Abel} has five invariant curves, namely, $1+c_1 p(t)x=0$ where $p(t)$ is one of the following polynomials: \[2t^3/3-2/3,\quad t^3+2 t^2 /3 -t/3 ,\quad t^3-2t^2/3 -t/3,\] \[t^3-t^2/3 - t +1/3,\quad \text{or} \quad t^3 +t^2/3 -t - 1/3.\]
The other cases are treated analogously.
\end{exam}

For $\deg(A) = 7$, in most cases our computer calculations
consumed all available memory before giving an answer.
However, we were able to obtain some interesting examples like the one below.

\begin{exam}
If 
\[A(t) = (t-z_1)(t-z_2)(t-z_2+3)(t-z_2+6)(t-z_2+8)(t-z_2+18)(3t-z_2-2z_1+18)\]
and 
\begin{align*} B(t) = &\  7t^3-(16 z_2+5z_1-123) t^2 + 
(11z_2^2+10z_1z_2-176z_2-70z_1+504) t - \\ & -2z_2^3-5z_1z_2^2+53z_2^2+70z_1 z_2-324z_2-180z_1+324,\end{align*} then \eqref{eq:Abel} has four linearly independent invariant curves if $z_1 = z_2 + 2$ and five (linearly dependent) invariant curves if $z_1 = z_2 - 9$. 
\end{exam}

To summarize, with the above examples we have shown that the upper bound given in Theorem \ref{theo:A} is not optimal (at least, for $\deg(A)=5$). Moreover, we have explicitly given families with $(\deg(A)+1)/2$ linearly independent invariant curves of \eqref{eq:Abel} for $\deg(A) = 1, 3, 5$ or $7$, thus showing families of equations for which the sufficient condition of Theorem \ref{theo:C} is not satisfied for each $\deg(A) \leq 7$.


\begin{thebibliography}{99}

\bibitem{Andrews}
{G.E.} Andrews, {The theory of partitions. Reprint of the 1976 original.} Cambridge Mathematical Library. Cambridge University Press, Cambridge, 1998. 

\bibitem{BNP} D.M. Benardete, V.W. Noonburg, B. Pollina, {Qualitative tools for studying periodic solutions and bifurcations as applied to the periodically harvested
	logistic equation,} Amer. Math. Monthly 115 (3) (2008) 202--219.

\bibitem{BFOS} 
J.L. Bravo, M. Fern\'andez, I. Ojeda, F. S\'anchez, {Uniqueness of limit cycles for quadratic vector fields,} Discrete Contin. Dyn. Syst. 39 (1) (2019) 483--502.

\bibitem{BFY1} M. Briskin, J.P. Fran\c{c}oise, Y. Yomdin, {Center conditions, compositions of polynomials and moments on algebraic curves,} Ergodic Theory Dyn. Syst. 19 (5) (1999) 1201--1220.

\bibitem{BFY2} M. Briskin, J.P. Fran\c{c}oise, Y. Yomdin, {Center conditions II: Parametric and model center problems}, Israel J. Math. 118 (2000) 61--82.

\bibitem{BRY} M. Briskin, N. Roytvarf, Y. Yomdin, {Center conditions at infinity for Abel differential equation}, Ann. of Math. 172 (2010) 437--483.

\bibitem{CGM17}
A. Cima, A. Gasull, F. Mañosas,
{On the number of polynomial solutions of Bernoulli and Abel polynomial differential equations,}
J. Differential Equations,
263 (11) (2017) 7099--7122.

\bibitem{Darboux}
G. Darboux, 
{Mémoire sur les équations différentielles algébriques du premier ordre et du premier degré},
Bull. Sci. Math, Série 2, Tome 2 (1878) 60--96.

\bibitem{Singular}
W. Decker, {G.-M.} Greuel,  G. Pfister, H. Sch{\"o}nemann,  
\newblock {\sc Singular} {4-2-0} --- {A} computer algebra system for polynomial computations.
\newblock {http://www.singular.uni-kl.de} (2019).

\bibitem{G} Gasull, A. Some open problems in low dimensional dynamical systems. SeMA J.  78 (2021) 233--269. 

\bibitem{GTZ16}	
A. Gasull, J. Torregrosa, X. Zhang,
{The number of polynomial solutions of polynomial Riccati equations},
J. Differential Equations,
261(9),
(2016),
5071--5093.

\bibitem{GGL13}
J. Gin\'e, T. Grau, J. Llibre, 
{On the polynomial limit cycles of polynomial differential equations}, 
Israel J. Math., 
106
(2013) 
481--507.

\bibitem{GS10}
J. Gin\'e, X. Santallusia, 
{Abel differential equations admitting a certain first integral}, 
J. Math. Anal. Appl., 370 (2010) 187--199.

\bibitem{LN} 
A. Lins Neto, 
{On the number of solutions of the equation $\frac{d x}{dt}=\sum_{j=0}^n a_j(t)x^j$, $0\leq t\leq 1$, for which $x(0)=x(1)$}, 
Inv. Math. {59} (1980) 67--76.

\bibitem{LV18}
J. Llibre, C. Valls, 
{Polynomial Solutions of Equivariant Polynomial Abel Differential Equations,} 
Adv. Nonlinear Stud., 
18(3) 
(2018) 
537--542. 

\bibitem{LLV} 
J. {Llibre}, C. Valls, 
{Rational limit cycles of Abel equations},
Commun. Pure Appl. Anal. 20(3) (2021) 1077--1089.

\bibitem{LLWW}	C. Liu, C. Li, X. Wang and J. Wu, {On the rational limit cycles of Abel equations}, Chaos Solitons Fractals, 110 (2018) 2--32.

\bibitem{OV20}
R. Oliveira, C. Valls,
{On the Abel differential equations of third kind},
Discrete Contin. Dyn. Syst. Ser. B , 25(5) (2019) 1821--1834.

\bibitem{P} 
F. Pakovich,
{Solution of the parametric Center Problem for Abel Equation},
J. Eur. Math. Soc. (JEMS), 19 (2017) 2343--2369.

\bibitem{RS}
V. G. Romanovski and D. S. Shafer, {The Centre and Cyclicity Problems. A Computational Algebra Approach}, Birkhäuser, 2009.

\bibitem{S} 
S.Smale, 
{Mathematical problems for the next century}, 
Math. Intelligencer, 20 (1998) 7--15.              

\bibitem{V20}
C. Valls, 
{On the Polynomial Solutions and Limit Cycles of Some Generalized Polynomial Ordinary Differential Equations},
Mathematics, 8(7) (2020) 1139.
	    
\bibitem{V20r}	    
C. Valls,  
{Rational Limit Cycles on Abel Polynomial Equations,} 
Mathematics, 8(6) (2020) 885.	    



\end{thebibliography}
\end{document}